\documentclass[12pt,reqno]{article}

\usepackage[usenames]{color}
\usepackage[colorlinks=true,
linkcolor=webgreen, filecolor=webbrown,
citecolor=webgreen]{hyperref}

\definecolor{webgreen}{rgb}{0,.5,0}
\definecolor{webbrown}{rgb}{.6,0,0}

\usepackage{amssymb}
\usepackage{graphicx}
\usepackage{amscd}
\usepackage{lscape}
\usepackage{tikz}
\usepackage{tikz-cd}
\usepackage{pgfplots}

\usetikzlibrary{matrix}
\usetikzlibrary{fit,shapes}
\usetikzlibrary{positioning, calc}
\tikzset{circle node/.style = {circle,inner sep=1pt,draw, fill=white},
        X node/.style = {fill=white, inner sep=1pt},
        dot node/.style = {circle, draw, inner sep=5pt}
        }
\usepackage{tkz-fct}

\usepackage{amsthm}
\newtheorem{theorem}{Theorem}

\newtheorem{proposition}[theorem]{Proposition}
\newtheorem{corollary}[theorem]{Corollary}

\theoremstyle{definition}

\newtheorem{example}[theorem]{Example}

\usepackage{float}

\usepackage{graphics,amsmath}
\usepackage{amsfonts}
\usepackage{latexsym}
\usepackage{epsf}

\newcommand{\seqnum}[1]{\href{http://oeis.org/#1}{\underline{#1}}}

\begin{document}

\begin{center}
\vskip 1cm{\LARGE\bf The $\gamma$-Vectors of Pascal-like Triangles Defined by Riordan Arrays} \vskip 1cm \large
Paul Barry\\
School of Science\\
Waterford Institute of Technology\\
Ireland\\
\href{mailto:pbarry@wit.ie}{\tt pbarry@wit.ie}
\end{center}
\vskip .2 in

\begin{abstract} We define and characterize the $\gamma$-matrix associated to Pascal-like matrices that are defined by ordinary and exponential Riordan arrays. We also define and characterize the $\gamma$-matrix of the reversions of these triangles, in the case of ordinary Riordan arrays. We are led to the $\gamma$-matrices of a one-parameter family of generalized Narayana triangles. Thus these matrices generalize the matrix of $\gamma$-vectors of the associahedron. The principal tools used are the bivariate generating functions of the triangles and Jacobi continued fractions.\end{abstract}

\section{Introduction}
A polynomial $P_n(x)=\sum_{k=0}^n a_{n,k}x^k$ of degree $n$ is said to be \emph{reciprocal} if
$$P_n(x)=x^n P_n(1/x).$$
Thus we have
$$[x^k] P_n(x)=a_{n,k}=[x^k] x^n P_n(1/x).$$
Now
\begin{eqnarray*}
[x^k] x^n P_n(1/x)&=& [x^{k-n}] \sum_{i=0} a_{n,i} \frac{1}{x^i}\\
&=& [x^{k-n}] \sum_{i=0} a_{n,i} x^{-i}\\
&=& a_{n,n-k}.\end{eqnarray*}
Thus $P_n(x)=\sum_{k=0}^n a_{n,k} x^k$ defines a family of reciprocal polynomials if and only if $a_{n,k}=a_{n,n-k}$.
We shall call a lower-triangular matrix $(a_{n,k})$  \emph{Pascal-like} if
\begin{enumerate}
\item $a_{n,k}=a_{n,n-k}$
\item $a_{n,0}=a_{n,n}=1.$
\end{enumerate}
Such a matrix will then be the coefficient array of a family of monic reciprocal polynomials.

We have the following well-known result \cite{Gal}
\begin{proposition} Let $P_n(x)$ be a reciprocal polynomial of degree $n$. Then there exists a unique polynomial $\gamma_n$ of degree $\lfloor \frac{n}{2} \rfloor$ with the property
$$P_n(x)=(1+x)^n \gamma_n\left(\frac{x}{(1+x)^2}\right).$$
If $P_n(x)$ has integer coefficients then so does $\gamma_n(x)$.
\end{proposition}
By this means, we can associate to every Pascal-like matrix $(a_{n,k})$ a matrix $(\gamma_{n,k})$ so that
for all $n$, we have
$$P_n(x)=\sum_{k=0}^n a_{n,k}x^k = \sum_{k=0}^{\lfloor \frac{n}{k} \rfloor} \gamma_{n,k}x^k (1+x)^{n-2k}.$$
We shall call this matrix the $\gamma$-matrix associated to the coefficient array $(a_{n,k})$ of the family of polynomials $P_n(x)$. 

We can characterize the matrix $(a_{n,k})$ in terms of the $\gamma$-matrix $(\gamma_{n,k})$ as follows. Before we do this, we shall change our notation somewhat. In algebraic topology, it is customary to use the notation $h(x)$ for palindromic (reciprocal) polynomials \cite{Petersen, Stanley}. Thus we shall set $h_n(x)=\sum_{k=0}^n h_{n,k}x^k$, where $(h_{n,k})$ now denotes a Pascal-like matrix. We shall denote by $h(x,y)$ the bivariate generating function of this matrix.
\begin{proposition}
For a Pascal-like matrix $(h_{n,k})$ we have
$$h_{n,k}= \sum_{i=0}^{\lfloor \frac{n}{2} \rfloor} \binom{n-2i}{k-i}\gamma_{n,i}.$$
\end{proposition}
\begin{proof}
We have
\begin{eqnarray*}
h_{n,k}&=& [x^k] \sum_{i=0}^n h_{n,i} x^i\\
&=& [x^k] \sum_{i=0}^{\lfloor \frac{n}{2} \rfloor} \gamma_{n,i}x^i(1+x)^{n-2i}\\
&=& \sum_{i=0}^{\lfloor \frac{n}{2} \rfloor}\gamma_{n,i} [x^k] x^i (1+x)^{n-2i}\\
&=& \sum_{i=0}^{\lfloor \frac{n}{2} \rfloor}\gamma_{n,i} [x^{k-i}] \sum_{j=0}^{n-2i} \binom{n-2i}{j}x^j\\
&=& \sum_{i=0}^{\lfloor \frac{n}{2} \rfloor}\gamma_{n,i} \binom{n-2i}{k-i}.\end{eqnarray*}
\end{proof}
\begin{example} The identity
$$\binom{n}{k}=\sum_{i=0}^{\lfloor \frac{n}{2} \rfloor} \binom{n-2i}{k-i}\delta_{i,0}$$ shows that the matrix that begins
$$\left(
\begin{array}{ccccccc}
 1 & 0 & 0 & 0 & 0 & 0 & 0 \\
 1 & 0 & 0 & 0 & 0 & 0 & 0 \\
 1 & 0 & 0 & 0 & 0 & 0 & 0 \\
 1 & 0 & 0 & 0 & 0 & 0 & 0 \\
 1 & 0 & 0 & 0 & 0 & 0 & 0 \\
 1 & 0 & 0 & 0 & 0 & 0 & 0 \\
 1 & 0 & 0 & 0 & 0 & 0 & 0 \\
\end{array}
\right)$$ is the $\gamma$-matrix for the binomial matrix $\mathbf{B}=(\binom{n}{k})$ \seqnum{A007318}. Here, we have used the $Annnnnn$ number of the On-Line Encyclopedia of Integer Sequences \cite{SL1, SL2} for the binomial matrix (Pascal's triangle).
\end{example}
When $(\gamma_{n,k})$ is the $\gamma$-matrix for $(h_{n,k})$, we shall say the $(\gamma_{n,k})$ \emph{generates}, or \emph{is the generator of}, the matrix $(h_{n,k})$.
\begin{example} The matrix that begins
$$\left(
\begin{array}{ccccccc}
 1 & 0 & 0 & 0 & 0 & 0 & 0 \\
 1 & 0 & 0 & 0 & 0 & 0 & 0 \\
 1 & 1 & 0 & 0 & 0 & 0 & 0 \\
 1 & 0 & 0 & 0 & 0 & 0 & 0 \\
 1 & 0 & 1 & 0 & 0 & 0 & 0 \\
 1 & 0 & 0 & 0 & 0 & 0 & 0 \\
 1 & 0 & 0 & 1 & 0 & 0 & 0 \\
\end{array}
\right)$$ with $\gamma_{n,0}=1$, $\gamma_{n,\lfloor{\frac{n}{2}\rfloor}}=1$, and $0$ otherwise, generates the matrix $(h_{n,k})$ that begins
$$\left(
\begin{array}{cccccc}
 1 & 0 & 0 & 0 & 0 & 0 \\
 1 & 1 & 0 & 0 & 0 & 0 \\
 1 & 3 & 1 & 0 & 0 & 0 \\
 1 & 3 & 3 & 1 & 0 & 0 \\
 1 & 4 & 7 & 4 & 1 & 0 \\
 1 & 5 & 10 & 10 & 5 & 1 \\
\end{array}
\right).$$
\end{example}

\section{Pascal-like matrices defined by Riordan arrays}
We now wish to characterize the $\gamma$-matrices that are generators for the family of Pascal-like matrices that are determined by the one-parameter family of  Riordan arrays
$$\left(\frac{1}{1-x}, \frac{x(1+rx)}{1-x}\right).$$
We shall also determine the (generalized) $\gamma$-matrices associated to the reversion of these triangles.
We recall that an ordinary Riordan array $(g(x), f(x))$  is defined \cite{Book, Survey, SGWW} by two power series
$$g(x)=1+g_1 x+ g_2 x^2+ \ldots,$$
$$f(x)=x+ f_2 x^2+ f_3x^3+ \ldots,,$$ where the $(n,k)$-th element of the resulting lower-triangular matrix is given by $$a_{n,k}=[x^n] g(x)f(x)^k.$$
Such matrices are invertible. When they have integer entries, the inverse again is an integer matrix (note that we have $a_{n,n}=1$ in our case because $g_0=1$ and $f_1=1$). The bivariate generating function of the Riordan array $(g,f)$ is given by
$$\frac{g(x)}{1-yf(x)}.$$ Matrices defined in a similar manner but with $f(x)$ replaced by $\phi(x)=x^2+\phi_3 x^3 +\ldots$ are called ``stretched'' Riordan arrays \cite{Corsani}. They are not invertible but they do possess left inverses.
\begin{example} The stretched Riordan array $\left(\frac{1}{1-x}, x^2\right)$ begins
$$\left(
\begin{array}{ccccccc}
 1 & 0 & 0 & 0 & 0 & 0 & 0 \\
 1 & 0 & 0 & 0 & 0 & 0 & 0 \\
 1 & 1 & 0 & 0 & 0 & 0 & 0 \\
 1 & 1 & 0 & 0 & 0 & 0 & 0 \\
 1 & 1 & 1 & 0 & 0 & 0 & 0 \\
 1 & 1 & 1 & 0 & 0 & 0 & 0 \\
 1 & 1 & 1 & 1 & 0 & 0 & 0 \\
\end{array}
\right).$$
It is the $\gamma$-matrix for the Pascal-like triangle that begins
$$\left(
\begin{array}{ccccccc}
 1 & 0 & 0 & 0 & 0 & 0 & 0 \\
 1 & 1 & 0 & 0 & 0 & 0 & 0 \\
 1 & 3 & 1 & 0 & 0 & 0 & 0 \\
 1 & 4 & 4 & 1 & 0 & 0 & 0 \\
 1 & 5 & 9 & 5 & 1 & 0 & 0 \\
 1 & 6 & 14 & 14 & 6 & 1 & 0 \\
 1 & 7 & 20 & 29 & 20 & 7 & 1 \\
\end{array}
\right).$$
\end{example}
\begin{example} The matrix $\binom{n-k}{k}$ is the stretched Riordan array $\left(\frac{1}{1-x}, \frac{x^2}{1-x}\right)$ that begins
$$\left(
\begin{array}{ccccccc}
 1 & 0 & 0 & 0 & 0 & 0 & 0 \\
 1 & 0 & 0 & 0 & 0 & 0 & 0 \\
 1 & 1 & 0 & 0 & 0 & 0 & 0 \\
 1 & 2 & 0 & 0 & 0 & 0 & 0 \\
 1 & 3 & 1 & 0 & 0 & 0 & 0 \\
 1 & 4 & 3 & 0 & 0 & 0 & 0 \\
 1 & 5 & 6 & 1 & 0 & 0 & 0 \\
\end{array}
\right).$$
It generates the Pascal-like matrix that begins
$$\left(
\begin{array}{ccccccc}
 1 & 0 & 0 & 0 & 0 & 0 & 0 \\
 1 & 1 & 0 & 0 & 0 & 0 & 0 \\
 1 & 3 & 1 & 0 & 0 & 0 & 0 \\
 1 & 5 & 5 & 1 & 0 & 0 & 0 \\
 1 & 7 & 13 & 7 & 1 & 0 & 0 \\
 1 & 9 & 25 & 25 & 9 & 1 & 0 \\
 1 & 11 & 41 & 63 & 41 & 11 & 1 \\
\end{array}
\right).$$ We shall see that this is the Riordan array $\left(\frac{1}{1-x}, \frac{x(1+x)}{1-x}\right)$,  which is \seqnum{A008288}, the triangle of Delannoy numbers.
\end{example}
The bivariate generating function of the stretched Riordan array $(g(x), \phi(x))$ is given by
$$\frac{g(x)}{1-y \phi(x)}.$$

We have the following proposition \cite{Pas}.
\begin{proposition} The Riordan array $\left(\frac{1}{1-x}, \frac{x(1+rx)}{1-x}\right)$ is Pascal-like (for any $r \in \mathbb{Z}$). 
\end{proposition}
This is clear since in this case we have
$$h_{n,k}=\sum_{j=0}^k \binom{k}{j}\binom{n-j}{n-k-j}r^j=\sum_{j=0}^k \binom{k}{j}\binom{n-k}{n-k-j}(r+1)^j.$$
We can now characterize the $\gamma$-matrices that generate these Pascal-like matrices.
\begin{proposition} The $\gamma$-matrices that generate the Pascal-like matrices $\left(\frac{1}{1-x}, \frac{x(1+rx)}{1-x}\right)$ defined by ordinary Riordan arrays are given by the stretched Riordan arrays 
$$\left(\frac{1}{1-x},\frac{rx^2}{1-x}\right),$$ with $(n,k)$-th term 
$$\gamma_{n,k}=\binom{n-k}{k}r^k.$$
\end{proposition}
\begin{proof}
The generating function of the Pascal-like matrix $\left(\frac{1}{1-x}, \frac{x(1+rx)}{1-x}\right)$ is given by
$$h(x,y)=\frac{1}{1-x} \frac{1}{1-y \frac{x(1+rx)}{1-x}}=\frac{1}{1-(1+y)x-rx^2 y}.$$
Similarly, the generating function of the matrix $(\binom{n-k}{k}r^k)$ is given by
$$\gamma(x,y)=\frac{1}{1-x} \frac{1}{1-y\frac{rx^2}{1-x}}=\frac{1}{1-x-rx^2y}.$$
We now have
$$h(x,y)=\gamma\left((1+y)x, \frac{y}{(1+y)^2}\right).$$
\end{proof}

We recall that for a generating function $f(x)$, its INVERT($\alpha$) transform is the generating function
$$\frac{f(x)}{1+\alpha x f(x)}.$$
Note that
$$ \frac{\frac{v}{1+ \alpha x v}}{1-\alpha x \frac{v}{1+ \alpha x v}}=v,$$ and thus the inverse of the INVERT($\alpha$) transform is the INVERT($-\alpha$) transform.

\begin{corollary} The generating function $h(x,y)$ of the Pascal-like matrix $\left(\frac{1}{1-x}, \frac{x(1+rx)}{1-x}\right)$ is the INVERT($y$) transform of the generating function $\gamma(x,y)$ of the corresponding $\gamma$-matrix.
\end{corollary}
\begin{proof}
A direct calculation shows that for $\gamma(x,y)=\frac{1}{1-x-rx^2y}$ we have
$$\frac{\gamma(x,y)}{1- yx \gamma(x,y)}=\frac{1}{1-(y+1)x-rx^2y}=h(x,y).$$
\end{proof}
Equivalently, we can say that the generating function of the $\gamma$-matrix is the INVERT($-y$) transform of the generating function of the corresponding Pascal-like matrix.

We make the following observation, which will be relevant when we discuss a family of generalized Narayana triangles. The $\gamma$-matrix corresponding to the signed Pascal-like matrix
$$\left(\frac{1}{1+x}, \frac{-x(1+rx)}{1+x}\right)$$ has generating function
$$\frac{1}{1+x+rx^2y}.$$
This is the matrix with general term $(-1)^{n-k} r^k \binom{n-k}{k}$.
By a signed Pascal-like matrix in this case we mean that $a_{n,k}=a_{n,n-k}$ but we now have
$a_{n,0}=a_{n,n}=(-1)^n$.

We close this section by recalling the formula 
$$\gamma_n = (1+x)^n \gamma_n\left(\frac{x}{(1+x)^2}\right).$$ 
We now note that the inverse of the Riordan array 
$$\left(1, \frac{x}{(1+x)^2}\right)$$ is given by 
$$\left(1, xc(x)^2\right),$$   where 
$$c(x)=\frac{1-\sqrt{1-4x}}{2x}$$ is the generating function of the Catalan numbers $C_n=\frac{1}{n+1}\binom{2n}{n}$ \seqnum{A000108}. 
In fact, we have the following result \cite{Petersen}.
\begin{proposition} (\textbf{Zeilberger's Lemma}). 
$$\gamma_{n,k}=[x^k] \frac{h_n(xc(x)^2)}{c(x)^n}.$$ 
\end{proposition}
We can use this result to find an explicit formula for $\gamma_{n,k}$ in terms of $h_{n,k}$. 
We let $\alpha_{n,k}$ be the general $(n,k)$-th element of the Riordan array $(1, xc(x)^2)$.  We have 
$$\alpha_{n,k}=\binom{2n-1}{n-k}\frac{2k+0^{n+k}}{n+k+0^{n+k}}.$$ 
We let $\beta_{n,k}$ be the general $(n,k)$-th term of the Riordan array $\left(1, \frac{x}{c(x)}\right)$. 
We have $\beta_{n,n}=1$,  and 
$$\beta_{n,k}=\sum_{j=0}^{n-k} \frac{(-1)^j j}{n-k}\binom{k+j-1}{j}\binom{2(n-k)}{n-k-j},$$ otherwise. This is essentially \seqnum{A271875}. 
Then we have the following result.
\begin{corollary}
We have 
$$\gamma_{n,k}=\sum_{i=0}^k \left(\sum_{j=0}^n h_{n,j}\alpha_{i,j}\right)\beta_{n+k-i,n}.$$
\end{corollary}
\begin{proof} We have 
\begin{eqnarray*}
[x^k][x^k] \frac{h_n(xc(x)^2)}{c(x)^n}&=&\sum_{i=0}^n [x^i] \sum_{j=0}^n h_{n,j}(xc(x^2))^j [x^{k-i}] \frac{1}{c(x)^n}\\
&=& \sum_{i=0}^k\left(\sum_{j=0}^n h_{n,j}[x^i](xc(x)^2)^j\right)[x^{k-1+n}]\frac{x^n}{c(x)^n}\\
&=& \sum_{i=0}^k \left(\sum_{j=0}^n h_{n,j}\alpha_{i,j}\right)\beta_{n+k-i,n}.\end{eqnarray*}
\end{proof}
This gives us the following formula.
$$\gamma_{n,k}=\sum_{i=0}^k \sum_{j=0}^n h_{n,j}\binom{2i-1}{i-j}\frac{2j+0^{i+j}}{i+j+0^{i+j}} \text{If}\left(k=i,1,\sum_{m=0}^{k-i} \frac{m(-1)^m}{k-i} \binom{n-1+m}{m}\binom{2(k-i)}{k-i-m}\right).$$ 
\begin{example} If we take $(h_{n,k})$ to be the triangle of Eulerian numbers \seqnum{A008292} that begins
$$\left(
\begin{array}{ccccccc}
 1 & 0 & 0 & 0 & 0 & 0 & 0 \\
 1 & 1 & 0 & 0 & 0 & 0 & 0 \\
 1 & 4 & 1 & 0 & 0 & 0 & 0 \\
 1 & 11 & 11 & 1 & 0 & 0 & 0 \\
 1 & 26 & 66 & 26 & 1 & 0 & 0 \\
 1 & 57 & 302 & 302 & 57 & 1 & 0 \\
 1 & 120 & 1191 & 2416 & 1191 & 120 & 1 \\
\end{array}
\right)$$ we find that the $\gamma$-matrix $(\gamma_{n,k})$ is the triangle \seqnum{A101280} that begins 
$$\left(
\begin{array}{ccccccc}
 1 & 0 & 0 & 0 & 0 & 0 & 0 \\
 1 & 0 & 0 & 0 & 0 & 0 & 0 \\
 1 & 2 & 0 & 0 & 0 & 0 & 0 \\
 1 & 8 & 0 & 0 & 0 & 0 & 0 \\
 1 & 22 & 16 & 0 & 0 & 0 & 0 \\
 1 & 52 & 136 & 0 & 0 & 0 & 0 \\
 1 & 114 & 720 & 272 & 0 & 0 & 0 \\
\end{array}
\right).$$ 
This is the triangle of $\gamma$-vectors for the permutahedra (of type $A$). It also gives the number of permutations of $n$ objects with $k$ descents such that every descent is a peak \cite{Shapiro}.
\end{example}
\begin{example} 
We consider the Pascal-like matrix $(h_{n,k})=\left(\frac{1}{1-x},x\right)$ that begins 
$$\left(
\begin{array}{ccccccc}
 1 & 0 & 0 & 0 & 0 & 0 & 0 \\
 1 & 1 & 0 & 0 & 0 & 0 & 0 \\
 1 & 1 & 1 & 0 & 0 & 0 & 0 \\
 1 & 1 & 1 & 1 & 0 & 0 & 0 \\
 1 & 1 & 1 & 1 & 1 & 0 & 0 \\
 1 & 1 & 1 & 1 & 1 & 1 & 0 \\
 1 & 1 & 1 & 1 & 1 & 1 & 1 \\
\end{array}
\right).$$ We note that the row elements are constant. We have that 
$$\gamma_{n,k}=\sum_{i=0}^k \sum_{j=0}^n \binom{2i-1}{i-j}\frac{2j+0^{i+j}}{i+j+0^{i+j}} \text{If}\left(k=i,1,\sum_{m=0}^{k-i} \frac{m(-1)^m}{k-i} \binom{n-1+m}{m}\binom{2(k-i)}{k-i-m}\right).$$
We find that the $\gamma$-matrix in this case begins
$$\left(
\begin{array}{ccccccc}
 1 & 0 & 0 & 0 & 0 & 0 & 0 \\
 1 & 0 & 0 & 0 & 0 & 0 & 0 \\
 1 & -1 & 0 & 0 & 0 & 0 & 0 \\
 1 & -2 & 0 & 0 & 0 & 0 & 0 \\
 1 & -3 & 1 & 0 & 0 & 0 & 0 \\
 1 & -4 & 3 & 0 & 0 & 0 & 0 \\
 1 & -5 & 6 & -1 & 0 & 0 & 0 \\
\end{array}
\right).$$ 
This is the matrix $\left(\binom{n-k}{k}(-1)^k\right)$. Thus
$$\sum_{i=0}^{\lfloor \frac{n}{2} \rfloor} \binom{n-2i}{k-i}\binom{n-i}{i}(-1)^i = \text{If}[k \le n,1,0].$$ 
\end{example}
\section{Stretched Riordan arrays as $\gamma$-matrices}
Every stretched Riordan array of the form 
$$\left(\frac{1}{1-x}, x^2g(x)\right),$$ where 
$$g(x)=1+g_1x+g_2x^2+\cdots$$ can be used to generate a Pascal-like matrix. Thus to each power series $g(x)$ above we can associate a Pascal-like matrix whose $\gamma$-matrix is given by this stretched Riordan array. 

In this section, we shall concentrate on the case when $g(x)=\frac{1+rx}{1-x}$. 
\begin{example}
For $r=1$, we obtain the $\gamma$-matrix that begins 
$$\left(
\begin{array}{ccccccc}
 1 & 0 & 0 & 0 & 0 & 0 & 0 \\
 1 & 0 & 0 & 0 & 0 & 0 & 0 \\
 1 & 1 & 0 & 0 & 0 & 0 & 0 \\
 1 & 3 & 0 & 0 & 0 & 0 & 0 \\
 1 & 5 & 1 & 0 & 0 & 0 & 0 \\
 1 & 7 & 5 & 0 & 0 & 0 & 0 \\
 1 & 9 & 13 & 1 & 0 & 0 & 0 \\
\end{array}
\right).$$ 
The corresponding Pascal-like matrix then begins 
$$\left(
\begin{array}{ccccccc}
 1 & 0 & 0 & 0 & 0 & 0 & 0 \\
 1 & 1 & 0 & 0 & 0 & 0 & 0 \\
 1 & 3 & 1 & 0 & 0 & 0 & 0 \\
 1 & 6 & 6 & 1 & 0 & 0 & 0 \\
 1 & 9 & 17 & 9 & 1 & 0 & 0 \\
 1 & 12 & 36 & 36 & 12 & 1 & 0 \\
 1 & 15 & 64 & 101 & 64 & 15 & 1 \\
\end{array}
\right).$$ 
The row sums of this matrix, which begin 
$$1, 2, 5, 14, 37, 98, 261,\ldots$$ give \seqnum{A077938}, with generating function 
$$\frac{1}{1-2x-x^2-2x^3}.$$ 
The diagonal sums, which begin 
$$1, 1, 2, 4, 8, 16, 31,\ldots$$ are the Pentanacci numbers \seqnum{A001591} with generating function 
$$\frac{1}{1-x-x^2-x^3-x^4-x^5}.$$ 
\end{example}
We have the following proposition.
\begin{proposition} The Pascal-like triangle that begins 
$$\left(
\begin{array}{ccccccc}
 1 & 0 & 0 & 0 & 0 & 0 & 0 \\
 1 & 1 & 0 & 0 & 0 & 0 & 0 \\
 1 & 3 & 1 & 0 & 0 & 0 & 0 \\
 1 & r+5 & r+5 & 1 & 0 & 0 & 0 \\
 1 & 2 r+7 & 4 r+13 & 2 r+7 & 1 & 0 & 0 \\
 1 & 3 r+9 & 11 r+25 & 11 r+25 & 3 r+9 & 1 & 0 \\
 1 & 4 r+11 & r^2+22 r+41 & 2 r^2+36 r+63 & r^2+22 r+41 & 4 r+11 & 1 \\
\end{array}
\right)$$
with $\gamma$-matrix given by the stretched Riordan array $\left(\frac{1}{1-x}, \frac{x^2(1+rx)}{1-x}\right)$, 
has row sums with generating function 
$$\frac{1}{1-2x-x^2-2rx^3},$$ and diagonal sums given by the generalized Pentanacci numbers with generating function
$$\frac{1}{1-x-x^2-x^3-rx^4-rx^5}.$$
\end{proposition}
\section{Reverting triangles}
Let $h(x,y)$ be the generating function of the lower-triangular matrix $h_{n,k}$, with $h_{0,0}=1$.
By the reversion of this triangle, we shall mean the triangle whose generating function $h^*(x,y)$ is given by
$$h^*(x,y)=\frac{1}{x} \text{Rev}_x (x h(x,y)).$$
Procedurally, this means that we solve the equation
$$u h(u,y)=x$$  and then we divide the solution $u(x,y)$ that satisfies $u(0,y)=0$ by $x$.
\begin{proposition} The generating function of the reversion of the Pascal-like matrix defined by the Riordan array $\left(\frac{1}{1-x},\frac{x(1+rx)}{1-x}\right)$ is given by
$$h^*(x,y)=\frac{1}{1+x(y+1)}c\left(\frac{-rx^2y}{(1+x(y+1))^2}\right),$$ where
$$c(x)=\frac{1-\sqrt{1-4x}}{2x}$$ is the generating function of the Catalan numbers $C_n=\frac{1}{n+1} \binom{2n}{n}$. (\seqnum{A000108}).
\end{proposition}
\begin{proof} Solving the equation
$$\frac{u}{1-u(y+1)-ru^2y}=x$$ gives us
$$h^*(x,y)=\frac{-1-x(y+1)+\sqrt{1+2x(y+1)+x^2(1+2y(2r+1)+y^2)}}{2rx^2y}.$$
Thus $$h^*(x,y)=\frac{1}{1+x(y+1)}c\left(\frac{-rx^2y}{(1+x(y+1))^2}\right).$$
\end{proof}
We note that we can now calculate an expression for the terms of the reverted triangle, since, using the language of Riordan arrays, we have
$$h^*(x,y)=\left(\frac{1}{1+y(x+1)}, \frac{-rx^2y}{(1+x(y+1))^2}\right)\cdot c(x).$$
\begin{proposition}
We have
$$[x^n][y^i]h^*(x,y)=h_{n,i}^*=\sum_{k=0}^{\lfloor \frac{n}{2} \rfloor} (-1)^n(-r)^k \binom{n}{2k}C_k \binom{n-2k}{i-k}.$$
The $\gamma$-matrix of the reverted triangle $(h_{n,k}^*)$ is given by
$$\gamma_{n,k}^*=(-1)^n (-r)^k \binom{n}{2k}C_k.$$
The $\gamma$-matrix $(\gamma_{n,k}^*)$ of the reverted triangle $(h_{n,k}^*)$ is the reversion of the triangle $\gamma_{n,k}$.
\end{proposition}
\begin{proof}
The expression for $h_{n,k}^*$ results from a direct calculation.
Reverting the expression $\gamma(x,y)=\frac{1}{1-x-rx^2y}$ in the sense above gives us
$$\gamma^*(x,y)=\frac{1}{1+x}c\left(\frac{-rx^2y}{(1+x)^2}\right),$$ from which we deduce the other statements.
\end{proof}
\begin{example} For $r=-1,0,1$, the triangles $(h_{n,k})$ begin, respectively,
$$\left(
\begin{array}{ccccccc}
 1 & 0 & 0 & 0 & 0 & 0 & 0 \\
 1 & 1 & 0 & 0 & 0 & 0 & 0 \\
 1 & 1 & 1 & 0 & 0 & 0 & 0 \\
 1 & 1 & 1 & 1 & 0 & 0 & 0 \\
 1 & 1 & 1 & 1 & 1 & 0 & 0 \\
 1 & 1 & 1 & 1 & 1 & 1 & 0 \\
 1 & 1 & 1 & 1 & 1 & 1 & 1 \\
\end{array}
\right),\left(
\begin{array}{ccccccc}
 1 & 0 & 0 & 0 & 0 & 0 & 0 \\
 1 & 1 & 0 & 0 & 0 & 0 & 0 \\
 1 & 2 & 1 & 0 & 0 & 0 & 0 \\
 1 & 3 & 3 & 1 & 0 & 0 & 0 \\
 1 & 4 & 6 & 4 & 1 & 0 & 0 \\
 1 & 5 & 10 & 10 & 5 & 1 & 0 \\
 1 & 6 & 15 & 20 & 15 & 6 & 1 \\
\end{array}
\right),$$ and
$$\left(
\begin{array}{ccccccc}
 1 & 0 & 0 & 0 & 0 & 0 & 0 \\
 1 & 1 & 0 & 0 & 0 & 0 & 0 \\
 1 & 3 & 1 & 0 & 0 & 0 & 0 \\
 1 & 5 & 5 & 1 & 0 & 0 & 0 \\
 1 & 7 & 13 & 7 & 1 & 0 & 0 \\
 1 & 9 & 25 & 25 & 9 & 1 & 0 \\
 1 & 11 & 41 & 63 & 41 & 11 & 1 \\
\end{array}
\right).$$  The corresponding reverted triangles $(h_{n,k}^*)$, are, respectively,
$$\left(
\begin{array}{ccccccc}
 1 & 0 & 0 & 0 & 0 & 0 & 0 \\
 -1 & -1 & 0 & 0 & 0 & 0 & 0 \\
 1 & 3 & 1 & 0 & 0 & 0 & 0 \\
 -1 & -6 & -6 & -1 & 0 & 0 & 0 \\
 1 & 10 & 20 & 10 & 1 & 0 & 0 \\
 -1 & -15 & -50 & -50 & -15 & -1 & 0 \\
 1 & 21 & 105 & 175 & 105 & 21 & 1 \\
\end{array}
\right),\left(
\begin{array}{ccccccc}
 1 & 0 & 0 & 0 & 0 & 0 & 0 \\
 -1 & -1 & 0 & 0 & 0 & 0 & 0 \\
 1 & 2 & 1 & 0 & 0 & 0 & 0 \\
 -1 & -3 & -3 & -1 & 0 & 0 & 0 \\
 1 & 4 & 6 & 4 & 1 & 0 & 0 \\
 -1 & -5 & -10 & -10 & -5 & -1 & 0 \\
 1 & 6 & 15 & 20 & 15 & 6 & 1 \\
\end{array}
\right),$$ and
$$\left(
\begin{array}{ccccccc}
 1 & 0 & 0 & 0 & 0 & 0 & 0 \\
 -1 & -1 & 0 & 0 & 0 & 0 & 0 \\
 1 & 1 & 1 & 0 & 0 & 0 & 0 \\
 -1 & 0 & 0 & -1 & 0 & 0 & 0 \\
 1 & -2 & -4 & -2 & 1 & 0 & 0 \\
 -1 & 5 & 10 & 10 & 5 & -1 & 0 \\
 1 & -9 & -15 & -15 & -15 & -9 & 1 \\
\end{array}
\right).$$
Note that for $r=-1$, the reverted triangle is $(-1)^n$ times the Narayana triangle \seqnum{A001263}.

The corresponding $\gamma$-matrices $(\gamma_{n,k})$ are given by, respectively,
$$\left(
\begin{array}{ccccccc}
 1 & 0 & 0 & 0 & 0 & 0 & 0 \\
 1 & 0 & 0 & 0 & 0 & 0 & 0 \\
 1 & -1 & 0 & 0 & 0 & 0 & 0 \\
 1 & -2 & 0 & 0 & 0 & 0 & 0 \\
 1 & -3 & 1 & 0 & 0 & 0 & 0 \\
 1 & -4 & 3 & 0 & 0 & 0 & 0 \\
 1 & -5 & 6 & -1 & 0 & 0 & 0 \\
\end{array}
\right), \left(
\begin{array}{ccccccc}
 1 & 0 & 0 & 0 & 0 & 0 & 0 \\
 1 & 0 & 0 & 0 & 0 & 0 & 0 \\
 1 & 0 & 0 & 0 & 0 & 0 & 0 \\
 1 & 0 & 0 & 0 & 0 & 0 & 0 \\
 1 & 0 & 0 & 0 & 0 & 0 & 0 \\
 1 & 0 & 0 & 0 & 0 & 0 & 0 \\
 1 & 0 & 0 & 0 & 0 & 0 & 0 \\
\end{array}
\right),$$ and
$$\left(
\begin{array}{ccccccc}
 1 & 0 & 0 & 0 & 0 & 0 & 0 \\
 1 & 0 & 0 & 0 & 0 & 0 & 0 \\
 1 & 1 & 0 & 0 & 0 & 0 & 0 \\
 1 & 2 & 0 & 0 & 0 & 0 & 0 \\
 1 & 3 & 1 & 0 & 0 & 0 & 0 \\
 1 & 4 & 3 & 0 & 0 & 0 & 0 \\
 1 & 5 & 6 & 1 & 0 & 0 & 0 \\
\end{array}
\right).$$
The corresponding reverted $\gamma$-matrices $(\gamma_{n,k}^*)$ are then, respectively,
$$\left(
\begin{array}{ccccccc}
 1 & 0 & 0 & 0 & 0 & 0 & 0 \\
 -1 & 0 & 0 & 0 & 0 & 0 & 0 \\
 1 & 1 & 0 & 0 & 0 & 0 & 0 \\
 -1 & -3 & 0 & 0 & 0 & 0 & 0 \\
 1 & 6 & 2 & 0 & 0 & 0 & 0 \\
 -1 & -10 & -10 & 0 & 0 & 0 & 0 \\
 1 & 15 & 30 & 5 & 0 & 0 & 0 \\
\end{array}
\right), \left(
\begin{array}{ccccccc}
 1 & 0 & 0 & 0 & 0 & 0 & 0 \\
 -1 & 0 & 0 & 0 & 0 & 0 & 0 \\
 1 & 0 & 0 & 0 & 0 & 0 & 0 \\
 -1 & 0 & 0 & 0 & 0 & 0 & 0 \\
 1 & 0 & 0 & 0 & 0 & 0 & 0 \\
 -1 & 0 & 0 & 0 & 0 & 0 & 0 \\
 1 & 0 & 0 & 0 & 0 & 0 & 0 \\
\end{array}
\right),$$ and
$$\left(
\begin{array}{ccccccc}
 1 & 0 & 0 & 0 & 0 & 0 & 0 \\
 -1 & 0 & 0 & 0 & 0 & 0 & 0 \\
 1 & -1 & 0 & 0 & 0 & 0 & 0 \\
 -1 & 3 & 0 & 0 & 0 & 0 & 0 \\
 1 & -6 & 2 & 0 & 0 & 0 & 0 \\
 -1 & 10 & -10 & 0 & 0 & 0 & 0 \\
 1 & -15 & 30 & -5 & 0 & 0 & 0 \\
\end{array}
\right).$$
\end{example}
It is interesting to represent the generating functions of the $(\gamma_{n,k}^*)$ and the $(h_{n,k}^*)$ triangles as Jacobi continued fractions. We have
\begin{proposition}
The generating function $h^*(x,y)$ can be expressed as the Jacobi continued fraction
$$\mathcal{J}(-(y+1),-(y+1),-(y+1),\ldots; -ry,-ry,-ry,\ldots).$$
The generating function $\gamma^*(x,y)$ can be expressed as the Jacobi continued fraction
$$\mathcal{J}(-1,-1,-1,\ldots; -ry,-ry,-ry,\ldots).$$
\end{proposition}
\begin{proof}
We solve the continued fraction equation
$$u=\frac{1}{1+(y+1)x+rx^2 u}$$ to retrieve the generating function $h^*(x,y)$.
Similarly, we solve the continued fraction equation
$$u=\frac{1}{1+x+rx^2 u}$$ to retrieve the generating function $\gamma^*(x,y)$.
\end{proof}
Note that we have used the notation $\mathcal{J}(a,b,c,\ldots;r,s,t,\ldots)$ to denote the Jacobi continued fraction \cite{CFT, Wall}
$$\cfrac{1}{1-a x-
\cfrac{r x^2}{1- b x-
\cfrac{s x^2}{1- c x-
\cfrac{t x^2}{1-\cdots}}}}.$$
We can now express the relationship between the generating functions $h^*(x,y)$ and $\gamma^*(x,y)$ in terms of repeated binomial transforms. 
\begin{corollary} The generating function $h^*(x,y)$ is the $(-y)$-th binomial transform of the $\gamma$ generating function $\gamma^*(x,y)$:
$$h^*(x,y)=\frac{1}{1+xy}\gamma^*\left(\frac{x}{1+xy},y\right).$$
Equivalently,  the $\gamma$ generating function $\gamma^*(x,y)$ is the $y$-th binomial transform of the generating function $h^*(x,y)$:
$$\gamma^*(x,y)=\frac{1}{1-xy}h^*\left(\frac{x}{1-xy},y\right).$$
\end{corollary}
This reflects the general assertion that the reversion of an INVERT transform is a binomial transform.

\section{The $\gamma$-vectors of generalized Narayana numbers}
The Riordan array $\left(\frac{1}{1+x}, \frac{-x(1+rx)}{1+x}\right)$, with bivariate generating function
$$\frac{1}{1+x(y+1)+rx^2y},$$ has a $\gamma$-matrix with generating function
$$\frac{1}{1+x+rx^2y}.$$
We shall call elements of the reversions of the Riordan array $\left(\frac{1}{1+x}, \frac{-x(1+rx)}{1+x}\right)$ $r$-Narayana numbers. The Narayana numbers $N_{n,k}=\frac{1}{k+1} \binom{n+1}{k}\binom{n}{k}$ are then the $1$-Narayana numbers.
The bivariate generating function for the $r$-Narayana numbers is given by
$$\frac{1}{1-x(y+1)}c\left(\frac{rx^2y}{(1-x(y+1))^2}\right).$$
The bivariate generating function for the $\gamma$-matrix of the $r$-Narayana numbers is then obtained by reverting the generating function $\frac{1}{1+x+rx^2y}$. We thus obtain the following result.
\begin{proposition}
The $\gamma$-matrix for the $r$-Narayana numbers has generating function
$$\frac{1}{1-x}c\left(\frac{rx^2y}{(1-x)^2}\right).$$
\end{proposition}
This is the matrix that begins
$$\left(
\begin{array}{ccccccc}
 1 & 0 & 0 & 0 & 0 & 0 & 0 \\
 1 & 0 & 0 & 0 & 0 & 0 & 0 \\
 1 & r & 0 & 0 & 0 & 0 & 0 \\
 1 & 3 r & 0 & 0 & 0 & 0 & 0 \\
 1 & 6 r & 2 r^2 & 0 & 0 & 0 & 0 \\
 1 & 10 r & 10 r^2 & 0 & 0 & 0 & 0 \\
 1 & 15 r & 30 r^2 & 5 r^3 & 0 & 0 & 0 \\
\end{array}
\right),$$ with general term
$$\binom{n}{2k}r^kC_k.$$
For $r=-1,0,1$, the matrices $\left(\frac{1}{1+x}, \frac{-x(1+rx)}{1+x}\right)$ begin, respectively,
$$\left(
\begin{array}{ccccccc}
 1 & 0 & 0 & 0 & 0 & 0 & 0 \\
 -1 & -1 & 0 & 0 & 0 & 0 & 0 \\
 1 & 3 & 1 & 0 & 0 & 0 & 0 \\
 -1 & -5 & -5 & -1 & 0 & 0 & 0 \\
 1 & 7 & 13 & 7 & 1 & 0 & 0 \\
 -1 & -9 & -25 & -25 & -9 & -1 & 0 \\
 1 & 11 & 41 & 63 & 41 & 11 & 1 \\
\end{array}
\right),\left(
\begin{array}{ccccccc}
 1 & 0 & 0 & 0 & 0 & 0 & 0 \\
 -1 & -1 & 0 & 0 & 0 & 0 & 0 \\
 1 & 2 & 1 & 0 & 0 & 0 & 0 \\
 -1 & -3 & -3 & -1 & 0 & 0 & 0 \\
 1 & 4 & 6 & 4 & 1 & 0 & 0 \\
 -1 & -5 & -10 & -10 & -5 & -1 & 0 \\
 1 & 6 & 15 & 20 & 15 & 6 & 1 \\
\end{array}
\right),$$ and
$$\left(
\begin{array}{ccccccc}
 1 & 0 & 0 & 0 & 0 & 0 & 0 \\
 -1 & -1 & 0 & 0 & 0 & 0 & 0 \\
 1 & 1 & 1 & 0 & 0 & 0 & 0 \\
 -1 & -1 & -1 & -1 & 0 & 0 & 0 \\
 1 & 1 & 1 & 1 & 1 & 0 & 0 \\
 -1 & -1 & -1 & -1 & -1 & -1 & 0 \\
 1 & 1 & 1 & 1 & 1 & 1 & 1 \\
\end{array}
\right).$$
The corresponding matrices of $r$-Narayana numbers are, respectively,
$$\left(
\begin{array}{ccccccc}
 1 & 0 & 0 & 0 & 0 & 0 & 0 \\
 1 & 1 & 0 & 0 & 0 & 0 & 0 \\
 1 & 1 & 1 & 0 & 0 & 0 & 0 \\
 1 & 0 & 0 & 1 & 0 & 0 & 0 \\
 1 & -2 & -4 & -2 & 1 & 0 & 0 \\
 1 & -5 & -10 & -10 & -5 & 1 & 0 \\
 1 & -9 & -15 & -15 & -15 & -9 & 1 \\
\end{array}
\right), \left(
\begin{array}{ccccccc}
 1 & 0 & 0 & 0 & 0 & 0 & 0 \\
 1 & 1 & 0 & 0 & 0 & 0 & 0 \\
 1 & 2 & 1 & 0 & 0 & 0 & 0 \\
 1 & 3 & 3 & 1 & 0 & 0 & 0 \\
 1 & 4 & 6 & 4 & 1 & 0 & 0 \\
 1 & 5 & 10 & 10 & 5 & 1 & 0 \\
 1 & 6 & 15 & 20 & 15 & 6 & 1 \\
\end{array}
\right),$$ and
$$\left(
\begin{array}{ccccccc}
 1 & 0 & 0 & 0 & 0 & 0 & 0 \\
 1 & 1 & 0 & 0 & 0 & 0 & 0 \\
 1 & 3 & 1 & 0 & 0 & 0 & 0 \\
 1 & 6 & 6 & 1 & 0 & 0 & 0 \\
 1 & 10 & 20 & 10 & 1 & 0 & 0 \\
 1 & 15 & 50 & 50 & 15 & 1 & 0 \\
 1 & 21 & 105 & 175 & 105 & 21 & 1 \\
\end{array}
\right).$$
This last matrix, as expected, is the Narayana triangle \seqnum{A001263}.
The corresponding $\gamma$-matrices for these $r$-Narayana triangles are, respectively,
$$\left(
\begin{array}{ccccccc}
 1 & 0 & 0 & 0 & 0 & 0 & 0 \\
 1 & 0 & 0 & 0 & 0 & 0 & 0 \\
 1 & -1 & 0 & 0 & 0 & 0 & 0 \\
 1 & -3 & 0 & 0 & 0 & 0 & 0 \\
 1 & -6 & 2 & 0 & 0 & 0 & 0 \\
 1 & -10 & 10 & 0 & 0 & 0 & 0 \\
 1 & -15 & 30 & -5 & 0 & 0 & 0 \\
\end{array}
\right), \left(
\begin{array}{ccccccc}
 1 & 0 & 0 & 0 & 0 & 0 & 0 \\
 1 & 0 & 0 & 0 & 0 & 0 & 0 \\
 1 & 0 & 0 & 0 & 0 & 0 & 0 \\
 1 & 0 & 0 & 0 & 0 & 0 & 0 \\
 1 & 0 & 0 & 0 & 0 & 0 & 0 \\
 1 & 0 & 0 & 0 & 0 & 0 & 0 \\
 1 & 0 & 0 & 0 & 0 & 0 & 0 \\
\end{array}
\right),$$ and
$$\left(
\begin{array}{ccccccc}
 1 & 0 & 0 & 0 & 0 & 0 & 0 \\
 1 & 0 & 0 & 0 & 0 & 0 & 0 \\
 1 & 1 & 0 & 0 & 0 & 0 & 0 \\
 1 & 3 & 0 & 0 & 0 & 0 & 0 \\
 1 & 6 & 2 & 0 & 0 & 0 & 0 \\
 1 & 10 & 10 & 0 & 0 & 0 & 0 \\
 1 & 15 & 30 & 5 & 0 & 0 & 0 \\
\end{array}
\right).$$  This last matrix is \seqnum{A055151}. The rows of this triangle are the $\gamma$-vectors of the $n$-dimensional (type A) associahedra \cite{Petersen}. We have seen that its elements are given by 
$$\gamma_{n,k}=\sum_{i=0}^k \sum_{j=0}^n N_{n,j}\binom{2i-1}{i-j}\frac{2j+0^{i+j}}{i+j+0^{i+j}} \text{If}\left(k=i,1,\sum_{m=0}^{k-i} \frac{m(-1)^m}{k-i} \binom{n-1+m}{m}\binom{2(k-i)}{k-i-m}\right),$$ where 
$N_{n,k}$ denotes the $(n,k)$-th Narayana number \seqnum{A001263}.

The relationship between the $\gamma$-matrix and the $r$-Narayana numbers can be further clarified as follows.

\begin{proposition}
The generating function of the $r$-Narayana numbers can be expressed as the Jacobi continued fraction
$$\mathcal{J}((y+1),(y+1),(y+1),\ldots; ry,ry,ry,\ldots).$$
The generating function of the corresponding $\gamma$-matrix can be expressed as the Jacobi continued fraction
$$\mathcal{J}(1,1,1,\ldots; ry,ry,ry,\ldots).$$
\end{proposition}
\begin{corollary} The generating function of the $r$-Narayana numbers is the $y$-th binomial transform of the generating function of the corresponding $\gamma$-matrix.
$$h^*(x,y)=\frac{1}{1-xy}\gamma^*\left(\frac{x}{1-xy},y\right).$$
Equivalently,  the $\gamma$ generating function $\gamma^*(x,y)$ is the $(-y)$-th binomial transform of the generating function $h^*(x,y)$:
$$\gamma^*(x,y)=\frac{1}{1+xy}h^*\left(\frac{x}{1+xy},y\right).$$
\end{corollary}

\section{Pascal-like triangles defined by exponential Riordan arrays}

We recall that an exponential Riordan array $[g(x), f(x)]$ \cite{Book, DeutschShap} is defined by two exponential generating functions
$$g(x)=1+ g_1 \frac{x}{1!}+g_2 \frac{x}{2!}+ \cdots,$$ and 
$$f(x)=\frac{x}{1!}+f_2 \frac{x^2}{2!}+\cdots,$$ with its $(n,k)$-th term $a_{n,k}$ given by 
$$a_{n,k}=\frac{n!}{k!} [x^n] g(x)f(x)^k.$$ 
In the context of Pascal-like matrices, we have that the exponential Riordan array
$$\left[e^x, x(1+rx/2)\right],$$ with general term
$$h_{n,k}=\frac{n!}{k!} \sum_{j=0}^k \frac{r^j}{(n-k-j)!2^j},$$ is a Pascal-like matrix \cite{Pas_Exp}. This matrix begins 
$$\left(
\begin{array}{ccccccc}
 1 & 0 & 0 & 0 & 0 & 0 & 0 \\
 1 & 1 & 0 & 0 & 0 & 0 & 0 \\
 1 & r+2 & 1 & 0 & 0 & 0 & 0 \\
 1 & 3r+3 & 3r+3 & 1 & 0 & 0 & 0 \\
 1 & 6 r+4 & 3r^2+12r+6 & 6 r+4 & 1 & 0 & 0 \\
 1 & 10 r+5 & 15r^2+30r+10 & 15r^2+30r+10 & 10 r+5 & 1 & 0 \\
 1 & 15 r+6 & 45r^2+60r+15 & 15r^3+90r^2+90r+20 & 45r^2+60r+15 & 15 r+6 & 1 \\
\end{array}
\right).$$
We have the following result.
\begin{proposition} The $\gamma$-matrix of the the Pascal-like exponential Riordan array $\left[e^x, x(1+rx/2)\right]$ is the matrix with general term
$$ \binom{n}{2k}r^k (2k-1)!!$$
\end{proposition}
In fact, the generating function of the exponential Riordan array $\left[e^x, x(1+rx/2)\right]$ is given by
$$\mathcal{J}(y+1,y+1,y+1,\ldots;ry,2ry,3ry,\ldots)$$ while that of its $\gamma$-matrix is given by
$$\mathcal{J}(1,1,1,\ldots;ry,2ry,3ry,\ldots).$$
\begin{proposition} The generating function of the $\gamma$-matrix of the Pascal-like exponential Riordan array
$\left[e^x, x(1+rx/2)\right]$ has generating function
$$e^{x(1+rxy/2)}.$$
\end{proposition}
\begin{proof}
By the theory of exponential Riordan arrays, the generating function of the Riordan array
$\left[e^x, x(1+rx/2)\right]$ is given by
$$e^x e^{xy(1+rx/2)}.$$
Taking the $(-y)$-th binomial transform of this, we obtain
$$e^{x(1+rxy/2)}.$$
\end{proof}
\begin{example} For $r=1$, we get the $\gamma$-matrix that begins
$$\left(
\begin{array}{ccccccc}
 1 & 0 & 0 & 0 & 0 & 0 & 0 \\
 1 & 0 & 0 & 0 & 0 & 0 & 0 \\
 1 & 1 & 0 & 0 & 0 & 0 & 0 \\
 1 & 3 & 0 & 0 & 0 & 0 & 0 \\
 1 & 6 & 3 & 0 & 0 & 0 & 0 \\
 1 & 10 & 15 & 0 & 0 & 0 & 0 \\
 1 & 15 & 45 & 15 & 0 & 0 & 0 \\
\end{array}
\right).$$ This is \seqnum{A100861}, the triangle of Bessel numbers that count the number of $k$-matchings of the complete graph $K(n)$. The corresponding Pascal-like matrix begins
$$\left(
\begin{array}{ccccccc}
 1 & 0 & 0 & 0 & 0 & 0 & 0 \\
 1 & 1 & 0 & 0 & 0 & 0 & 0 \\
 1 & 3 & 1 & 0 & 0 & 0 & 0 \\
 1 & 6 & 6 & 1 & 0 & 0 & 0 \\
 1 & 10 & 21 & 10 & 1 & 0 & 0 \\
 1 & 15 & 55 & 55 & 15 & 1 & 0 \\
 1 & 21 & 120 & 215 & 120 & 21 & 1 \\
\end{array}
\right).$$ This is \seqnum{A100862}, which counts the number of $k$-matchings of the corona $K'(n)$ of the complete graph $K(n)$ and the complete graph $K(1)$.
\end{example}
\begin{example}
For $r=2$, we obtain the $\gamma$-matrix that begins
$$\left(
\begin{array}{ccccccc}
 1 & 0 & 0 & 0 & 0 & 0 & 0 \\
 1 & 0 & 0 & 0 & 0 & 0 & 0 \\
 1 & 2 & 0 & 0 & 0 & 0 & 0 \\
 1 & 6 & 0 & 0 & 0 & 0 & 0 \\
 1 & 12 & 12 & 0 & 0 & 0 & 0 \\
 1 & 20 & 60 & 0 & 0 & 0 & 0 \\
 1 & 30 & 180 & 120 & 0 & 0 & 0 \\
\end{array}
\right).$$ This is \seqnum{A059344}, where row $n$ consists of the nonzero coefficients of the expansion of $2^n x^n$ in terms of Hermite polynomials with decreasing subscripts. The corresponding Pascal-like matrix begins
$$\left(
\begin{array}{ccccccc}
 1 & 0 & 0 & 0 & 0 & 0 & 0 \\
 1 & 1 & 0 & 0 & 0 & 0 & 0 \\
 1 & 4 & 1 & 0 & 0 & 0 & 0 \\
 1 & 9 & 9 & 1 & 0 & 0 & 0 \\
 1 & 16 & 42 & 16 & 1 & 0 & 0 \\
 1 & 25 & 130 & 130 & 25 & 1 & 0 \\
 1 & 36 & 315 & 680 & 315 & 36 & 1 \\de
\end{array}
\right).$$ The row sums of this matrix are given by \seqnum{A000898}, the number of symmetric involutions of $[2n]$ (Deutsch). 
\end{example}
\section{Conclusion} It is the case that the set of Pascal-like matrices defined by Riordan arrays is a restricted one. Nevertheless, we hope that this note indicates that they have interesting properties, including in particular their generating $\gamma$-matrices. In the case of Pascal-like matrices defined by ordinary Riordan arrays, we have seen that be reverting them, we find additional (signed) Pascal-like triangles, including triangles of Narayana type. The $\gamma$-matrices of these new triangles are again the reversions of the original triangles' $\gamma$-matrices. 

We have also shown that stretched Riordan arrays play a useful role, and in particular can lead to further (non-Riordan) Pascal-like matrices. We have also found it useful to use Riordan array techniques to find an explicit closed form formula for the elements $\gamma_{n,k}$ of the $\gamma$-matrix in terms of $h_{n,k}$.

\bigskip
\hrule
\bigskip
\noindent 2010 {\it Mathematics Subject Classification}: Primary
11B83; Secondary 33C45, 42C-5, 15B36, 15B05, 14N10, 11C20.
\noindent \emph{Keywords:} Gamma vector, Pascal-like triangle, Riordan array, Narayana number, Eulerian number, associahedron, permutahedron.

\bigskip
\hrule
\bigskip
\noindent (Concerned with sequences
\seqnum{A000108},
\seqnum{A000898},
\seqnum{A001263},
\seqnum{A001591},
\seqnum{A007318},
\seqnum{A008288},
\seqnum{A008292},
\seqnum{A055151},
\seqnum{A059344},
\seqnum{A077938},
\seqnum{A100861},
\seqnum{A100862}, 
\seqnum{A101280}, and
\seqnum{A271875}.)

\end{document}